\documentclass[11pt]{amsart}
\usepackage{mathrsfs,latexsym,amsfonts,amssymb}
\newtheorem{theorem}{Theorem}[section]
\newtheorem{lemma}[theorem]{Lemma}
\newtheorem{corollary}[theorem]{Corollary}
\newtheorem{question}[theorem]{Question}
\newtheorem{example}[theorem]{Example}
\theoremstyle{definition}
\newtheorem{definition}[theorem]{Definition}
\newtheorem{proposition}[theorem]{Proposition}
\theoremstyle{remark}
\newtheorem{remark}[theorem]{Remark}

\begin{document}
\title[Some topological properties of topological rough groups]
{Some topological properties of topological rough groups}

\author{Fucai Lin*}
\address{(Fucai Lin): School of mathematics and statistics,
Minnan Normal University, Zhangzhou 363000, P. R. China}
\email{linfucai2008@aliyun.com; linfucai@mnnu.edu.cn}

\author{Qianqian Sun}
\address{(Qianqian Sun): School of mathematics and statistics,
Minnan Normal University, Zhangzhou 363000, P. R. China}

\author{Yujin Lin}
\address{(Yujin Lin): School of mathematics and statistics,
Minnan Normal University, Zhangzhou 363000, P. R. China}

\author{Jinjin Li}
\address{(Jinjin Li): School of mathematics and statistics,
Minnan Normal University, Zhangzhou 363000, P. R. China}
\email{jinjinli@mnnu.edu.cn}

\thanks{The first author is supported by the NSFC (No. 11571158), the Natural Science Foundation of Fujian Province (No. 2017J01405) of China, the Program for New Century Excellent Talents in Fujian Province University, the Institute of Meteorological Big Data-Digital Fujian and Fujian Key Laboratory of Data Science and Statistics. The second author is supported by the Young and middle-aged project in Fujian Province (No. JAT190397). The fourth author is supported by the NSFC (No. 11871259).}

\thanks{*corresponding author}

\thanks{{\bf This paper is dedicated to professor Shou Lin on the occasion of his 60th birthday.}}


\begin{abstract}
Let $(U, R)$ be an approximation space with $U$ being non-empty set and $R$ being an equivalence relation on $U$, and let $\overline{G}$ and $\underline{G}$ be the upper approximation and the lower approximation of subset $G$ of $U$. A topological rough group $G$ is a rough group $G=(\underline{G}, \overline{G})$ endowed with a topology, which is induced from the upper approximation space $\overline{G}$, such that the product mapping $f: G\times G\rightarrow \overline{G}$ and the inverse mapping are continuous. In the class of topological rough groups, the relations of some separation axioms are obtained, some basic properties of the neighborhoods of the rough identity element and topological rough subgroups are investigated. In particular, some examples of topological rough groups are provided to clarify some facts about topological rough groups. Moreover, the version of open mapping theorem in the class of topological rough group is obtained. Further, some interesting open questions are posed.
\end{abstract}

\maketitle
\section{Introduction}
Rough set, as a mathematical theory for dealing with imprecise, uncertain and incomplete data, was first introduced by Pawlak \cite{P1982} in 1982. Its main idea is to use the known incomplete information or knowledge to approximately describe the concept of imprecise or uncertain, or to deal with ambiguous phenomena and problems according to the results of observation and measurement. After more than 30 years of research, the theory of rough set has been continuously improved and widely expanded in applications, see \cite{WM2019}. At present, it has been successfully applied in machine learning and knowledge discovery, information system analysis, data mining, decision support system, fault detection, process control, pattern recognition, etc.

In the past 30 years, the rough sets has been combined with some mathematical theories such
as topology and algebra \cite{AAAO2019,BIO2016,Bi1994,BO1995,DC2018,I1987,KM2006,LXL2012,PO1998,RK2002,S2011,WM2019,V20089,ZLL2019,ZLL2016}. The algebraic structures of rough sets are interesting topics, which have been
studied by many authors, for example, \cite{BO1995,I1987,PO1998}. In
1994, Biswas and Nanda \cite{Bi1994} introduced the notion of rough group and rough subgroups, which depends
on the upper approximation and does not depend on the lower approximation. Then many authors improved
the definitions of rough group and rough subgroup, and generalized the definitions of rough group and rough subgroup (such as, rough ideal, rough semigroup, etc) and prove some new properties, see \cite{AAAO2019,BO2016,CMN2007,Da2004,K1997,KW1996,LZ2014,WC2010,WC2013}. However, the definition of rough group has some defects, which lead to many gaps in the proofs of some papers. In 2011, Wu and Huang in \cite{WH2011} revised the definition of rough group that resolve the defects of the definition of rough group.

In \cite{BIO2016}, Bagirmaz et al. introduced the concept of topological rough group, and extended
the notion of a topological group to include algebraic structures of rough groups. In addition,
they presented properties and some examples. However, they use the concept of rough groups which was introduced in \cite{Bi1994}, which leads to some defects. In this paper, we study the topological rough groups base on the concept of rough groups which has been revised in \cite{WH2011}. We mainly discuss some topological properties of topological rough groups. In particular, we investigate the separation axioms, the neighborhoods of the rough identity element and the rough subgroups of topological rough groups. In Section 2, we introduce the necessary notation and terminologies which are
used for the rest of the paper. We also make some remarks about topological rough groups. In Section 3, we investigate the separation axioms for topological rough groups. We mainly provide some examples and prove some theorems to illustrate the relations of some separation axioms among topological rough groups, such as, $T_{0}$, $T_{1}$, $T_{2}$, etc. Section 4 is devoted to the study of some basic properties of topological rough groups. We mainly
discuss the neighborhoods of rough identity element of topological rough groups. In particular, we study the extremally disconnected topological rough groups. In Section 5, we investigate some basic properties of topological rough groups, and discuss that when the closure of a topological rough subgroup of a topological rough group is a topological rough subgroup. In Section 6, we redefine the concept of rough homomorphism, and prove the version of open mapping theorem in the class of topological rough groups.

\maketitle
\section{Preliminaries}
In this section, we introduce the necessary notation and terminologies. First of all, let $\mathbb{N}$, $\omega$, $\mathbb{Z}$ and $\mathbb{R}$ denote the sets of all positive
  integers, non-negative integers, all integers and all real numbers, respectively. Let $X$ be a topological space and let $A$ be a subset of $X$.
The \emph{closure} and \emph{interior} of $A$ are the smallest closed set containing $A$ and the largest open set contained in $A$ respectively, and are denoted by $A^{c}$ and $\mbox{Int}(A)$ respectively. For undefined
  notation and terminologies, the reader may refer to \cite{AA},
  \cite{E1989} and \cite{WM2019}.

First, we give the definition of rough groups introduced by Wu and Huang in 2011.

Let $(U, R)$ be an approximation space such that $U$ is any non-empty set and $R$ is an equivalence relation on $U$. We denote the equivalence class of object $x$ in $R$ by $[x]_{R}$. For a subset $X\subset U$, $$\overline{X}=\{x\in U: [x]_{R}\cap X\neq\emptyset\}$$ and $$\underline{X}=\{x\in U: [x]_{R}\subset X\}.$$are called {\it upper approximation} and {\it lower approximation} of $X$ in $(U, R)$, respectively.

Assume that ($\ast$) is a binary operation defined on $U$. We will use $xy$ instead of $x\ast y$ for all
composition of elements $x, y\in U$ as well as for composition of subsets $XY$, where $X, Y\subset U$.

\begin{definition}\cite{WH2011}
Let $S=(U, R)$ be an approximation space, and let ($\ast$) be a binary operation defined on $U$. A subset
$G$ of universe $U$ is called a {\it rough group} if the following properties are satisfied:
\begin{enumerate}
\smallskip
\item $\forall x, y\in G, xy\in \overline{G}$;
\smallskip\item $\forall x, y, z\in \overline{G}, (xy)z=x(yz)$;
\smallskip\item $\exists e\in \overline{G}$ such that $\forall x\in \overline{G}, ex=xe=x$, where $e$ is called the {\it rough identity element} of rough group $G$;
\smallskip\item $\forall x\in G, \exists y\in G$ such that $xy=yx=e$, where $y$ is called the {\it rough inverse element} of $x$ in $G$, we denote it by $x^{-1}$.
\end{enumerate}
\end{definition}

\begin{definition}\cite{Bi1994}
A non-empty rough subset $H=(\underline{H}, \overline{H})$ of a rough group $G=(\underline{G}, \overline{G})$ is
called a {\it rough subgroup} if it is a rough group itself.
\end{definition}

A necessary and sufficient condition for a subset $H$ of a rough group $G$ to be a rough subgroup is characterized as follows

\begin{theorem}\cite{Bi1994}
A rough subset $H$ is a rough subgroup of the rough group $G$ if the following two
conditions are satisfied:

\smallskip
(1) $\forall x, y\in H$, $xy\in \overline{H}$;

\smallskip
(2) $\forall y\in H$, $y^{-1}\in H$.
\end{theorem}

\begin{definition}\cite{MH2005}
A rough subgroup $N$ of rough group $G$ is called a {\it rough normal subgroup}, if for each $g\in G$, $gN=Ng$.
\end{definition}

\begin{definition}
Let $G$ be a rough group and $A\subset G$. We say that $A$ is {\it symmetric} if $A=A^{-1}$.
\end{definition}

\begin{definition}\cite{BIO2016}
A {\it topological rough group} is a rough group $(G, \ast)$ together with a topology $\tau$ on $\overline{G}$ satisfying the
following two properties:

\smallskip
(1) the mapping $f: G\times G\rightarrow \overline{G}$ defined by $f(x, y)=xy$ is continuous with respect to product topology on $G\times G$
and the topology $\tau_{G}$ on $G$ induced by $\tau$;

\smallskip
(2) the inverse mapping $g: G\rightarrow G$ defined by $g(x)=x^{-1}$ is continuous with respect to the topology $\tau_{G}$ on $G$
induced by $\tau$.
\end{definition}

\begin{proposition}\cite{BIO2016}
Let $G$ be a topological rough group. Then, every rough subgroup $H$ of $G$ with relative topology
is a topological rough subgroup.
\end{proposition}

Let $G$  be a topological rough group and let $\pi: G\times G\times G\rightarrow G^{3}$ be defined by $\pi(x, y, z)=xyz$ for any $x, y, z\in G$. Put $G^{3}_{0}=\pi^{-1}(G^{3}\cap \overline{G})$.

\begin{definition}
A {\it strongly topological rough group} is a topological rough group $G$ together with a topology $\tau$ on $\overline{G}$ satisfying the
following property ($\sharp$).

\smallskip
($\sharp$) The mapping $h: G_{0}^{3}\rightarrow \overline{G}$ defined by $h(x, y, z)=xyz$ is continuous with respect to the product topology of $G\times G\times G$ induced on $G_{0}^{3}$ and the topology $\tau_{G}$ on $G$ induced by $\tau$.
\end{definition}

\begin{remark}
 (1) Let $G$ be a topological rough group. Clearly, $G^{2}\subset \overline{G}$. If $G=G^{2}$ then it follows from the definition of topological rough group that $G$ is a topological group.

\smallskip
(2) Let $S=(U, R)$ be a topological group. If $G$ is a symmetric subset $G$ of $S$ such that $xy\in \overline{G}$ for any $x, y\in G$, then $G$ is a topological rough group.

\smallskip
(3) Each strongly topological rough group is a topological rough group. Indeed, there exists a topological rough group which is not a strongly topological rough group, see Example~\ref{e0}.

\smallskip
(4) There exists a strongly topological rough group which is not a topological group, see Example~\ref{e1}.

\smallskip
(5) Throughout this paper, we always assume the upper approximation space $\overline{G}$ with a topology $\tau$ and the topological rough group $G$ with the topology $\tau_{G}$ induced from $(\overline{G}, \tau)$. If $A$ is a subset of $G$, then we always denote the closure of $A$ in $(G, \tau_{G})$ and $(\overline{G}, \tau)$ by $A^{c}$ and $A^{c}_{\tau}$ respectively.

\smallskip
(6) Throughout this paper, the rough identity element of a topological rough group is denoted by $e$.
\end{remark}

\begin{example}\label{e1}
There exists a strongly topological rough group $G$ which is not a topological group
\end{example}

\begin{proof}
Let $U=\{\overline{0}, \overline{1}, \overline{2}\}$ be a set of surplus class with respect to module 3 and let ($\ast$) be the plus of surplus class. A classification of $U$ is $U/R$, where $E_{1}=\{\overline{0}, \overline{1}\}$ and $E_{2}=\{\overline{2}\}$. Let $G=\{\overline{1}, \overline{2}\}.$ Then $\overline{G}=U$. Then it is easy to see that $G$ is a strongly topological rough group which is not a topological group since the rough identity element $\overline{0}$ does not belong to $G$.
\end{proof}

\maketitle
\section{Separation axioms of topological rough groups}
In this section, we discuss the separation axioms of the topological rough groups. First of all, we recall some concepts.

Let $X$ be a topological space. Then

\smallskip
(1) $X$ is called a {\it $T_{0}$-space} \cite{E1989} if for every pair of distinct points $x, y\in X$ there exists an open set containing exactly one of these points.

\smallskip
(2) $X$ is called a {\it $T_{1}$-space} \cite{E1989} if for every pair of distinct points $x, y\in X$ there exists an open set $U\subset X$ such that $x\in U$ and $y\not\in U$.

\smallskip
(3) $X$ is called a {\it $T_{2}$-space} \cite{E1989}, or a {\it Hausdorff space}, if for every pair of distinct points $x, y\in X$ there exist open sets $U_{1}, U_{2}\subset X$ such that $x\in U_{1}, y\in U_{2}$ and $U_{1}\cap U_{2}=\emptyset.$

\smallskip
(4) $X$ is called a {\it $T_{3}$-space} \cite{E1989}, or a {\it regular space}, if $X$ is a $T_{1}$-space and for every $x\in X$ and every closed set $F\subset X$ such that $x\not\in F$ there exist open sets $U_{1}, U_{2}\subset X$ such that $x\in U_{1},F\subset U_{2}$ and $U_{1}\cap U_{2}=\emptyset.$

\smallskip
(5) $X$ is called a {\it $T_{3\frac{1}{2}}$-space} \cite{E1989}, or a {\it Tychonoff space}, or a {\it completely space}, if $X$ is a $T_{1}$-space and for every $x\in X$ and every closed set $F\subset X$ such that $x\not\in F$ there exists a continuous function $f: X\rightarrow I$ such that $f(x)=0$ and $f(y)=1$ for $y\in F$, where $I$ is the unit interval.

Clearly, we have $$T_{3\frac{1}{2}}\rightarrow T_{3}\rightarrow T_{2}\rightarrow T_{1}\rightarrow T_{0},$$and none of the above implications can be reversed.

It is well known that each $T_{0}$-topological group is completely regular. For topological rough group the situation with the separation axioms is worse.
The following example shows that there exists a $T_{0}$-topological rough group which is not a $T_{1}$-topological rough group.

\begin{example}\label{e0}
There exists a $T_{0}$-topological rough group $(G, \tau_{G})$ such that $e\not\in G$ and $\tau_{G}$ is not $T_{1}$.
\end{example}

\begin{proof}
Let $U=\{\overline{0}, \overline{1}, \overline{2}, \overline{3}, \overline{4}, \overline{5}\}$ be the set of surplus class with respect to module 6 and let ($\ast$) be the plus of surplus class. A classification of $U$ is $U/R$, where $$E_{1}=\{\overline{0}, \overline{1}, \overline{2}\}\ \mbox{and}\ E_{2}=\{\overline{3}, \overline{4}, \overline{5}\}.$$ Let $G=\{\overline{2}, \overline{3}, \overline{4}\}$ and $\tau=\{\emptyset, \overline{G}, \{\overline{2}\}, \{\overline{4}\}, \{\overline{2}, \overline{4}\}, \{\overline{2}, \overline{3}, \overline{4}\}\}$. Obviously, $G$ is a rough group. Since $\overline{G}=U$ and $\tau_{G}=\{\emptyset, \{\overline{2}\}, \{\overline{4}\}, \{\overline{2}, \overline{4}\}, \{\overline{2}, \overline{3}, \overline{4}\}\}$, it is easy to see that $(G, \tau_{G})$ is a topological rough group. Clearly, $\tau_{G}$ is $T_{0}$ and not $T_{1}$ since the set $\{\overline{4}\}$ is not closed in $\tau$.
\end{proof}

\begin{remark}
 (1) From the following Theorem~\ref{t0}, we see that the topological rough group $G$ in Example~\ref{e0} is not a strongly topological rough group.

\smallskip
(2) In Example~\ref{e0}, the space $\overline{G}$ is not $T_{0}$. If $\overline{G}$ is a $T_{0}$-space, then the separation axiom $T_{0}$ implies $T_{1}$ in $G$, see Theorem~\ref{t3}.
\end{remark}

\begin{theorem}\label{t3}
If $G$ is a topological rough group with $\overline{G}$ being $T_{0}$, then $G$ is $T_{1}$.
\end{theorem}

\begin{proof}
It suffices to prove that each point of $G$ is closed in $G$. Assume there exists a point $g\in G$ such that $\{g\}$ is not closed in $G$, then there exists a point $x\in G\setminus\{g\}$ such that $x\in \{g\}^{c}$. Hence, for any open neighborhood $V_{x}$ of $x$ in $G$, we have $g\in V_{x}$, then $g^{-1}$ belongs to the open neighborhood $V_{x}^{-1}$ of $x^{-1}$ since the inverse mapping is a homeomorphism. Clearly, $x^{-1}g\neq e$ and each open neighborhood of $(x^{-1}, g)$ in the product space $G\times G$ must contain the point $(g^{-1}, g)$. Because $f: G\times G\rightarrow \overline{G}$ is continuous at the point $(x^{-1}, g)$, it follows that each open neighborhood of $x^{-1}g$ in $\overline{G}$ contains the point $e$, hence $x^{-1}g\in \{e\}^{c}_{\tau}$. Moreover, Because $f: G\times G\rightarrow \overline{G}$ is continuous at the point $(g^{-1}, g)$, each open neighborhood of $e$ in $\overline{G}$ contains the point $x^{-1}g$, thus $e\in\{x^{-1}g\}^{c}_{\tau}$. Hence $x^{-1}g\in \{e\}^{c}_{\tau}$ and $e\in\{x^{-1}g\}^{c}_{\tau}$. However, $\overline{G}$ is $T_{0}$, which is a contradiction.
\end{proof}

In Example~\ref{e0}, the rough identity element $e$ does not belong to $G$. If $e\in G$, does the separation axiom $T_{0}$ imply $T_{1}$ in $G$? Indeed, we have the following example which gives a negative answer to this question.

\begin{example}\label{e2}
There exists a $T_{0}$-topological rough group $G$ such that $e\in G$ and $\tau_{G}$ is not $T_{1}$.
\end{example}

\begin{proof}
Let $U=\{1, 2, 3, 4, 5, 6, 7, 8, 9, 10\}$ and let ($\ast$) be the binary operation on $U$ as follows:
\[a\ast b=\left\{
\begin{array}{lll}
a\times b, & \mbox{if}\ a\times b<11,\\
a\times b\ (\mbox{mod}\ 11), & \mbox{if } a\times b\geq 11.\end{array}\right.\]
 A classification of $U$ is $U/R$, where $E_{1}=\{1, 2, 5\}$, $E_{2}=\{3, 8, 9\}$ and $E_{2}=\{4, 6, 7, 10\}$. Let $G=\{1, 2, 5, 6, 9\}$. Then $\overline{G}=U$. Let $$\mathscr{B}=\{\{1, 7, 8\}, \{3, 4, 10\}, \{2\}, \{6\}, \{2, 5\}, \{6, 9\}\}$$ be a base of the topology $\tau$ on $\overline{G}=U$. It is easy to see that $\tau_{G}$ has a base $$\mathscr{B}_{G}=\{\{1\}, \{2\}, \{6\}, \{2, 5\}, \{6, 9\}\}.$$ Obviously, $G$ is a $T_{0}$-space and $e\in G$. However, $G$ is not $T_{1}$ since the set $\{2\}$ is not closed in $G$.
\end{proof}

In Example~\ref{e2}, we see that the set $\{e\}$ is open and closed in $G$. Indeed, we have the following Corollary~\ref{c0}. First, we give some propositions, which are interesting themselves.

\begin{proposition}\label{p0}
Let $G$ be a $T_{0}$-topological rough group. If $e\in G$ then $\{e\}$ is closed in $G$.
\end{proposition}

\begin{proof}
Assume $\{e\}$ is not closed, then there exists a point $g\in G\setminus\{e\}$ such that $g\in\{e\}^{c}$. Since the inverse mapping is a homeomorphism, then it is easy to see that $g^{-1}\in\{e\}^{c}$. Hence $e\in V$ and $e\in W$ for any open neighborhoods $V$ and $W$ of $g$ and $g^{-1}$ in $G$, respectively. Because $f: G\times G\rightarrow \overline{G}$ is continuous at the point $(g, g^{-1})$, it is easy to see that $\{g, g^{-1}\}\subset U$ for any open neighborhood $U$ of $e$ in $G$. However, since $G$ is $T_{0}$ and $\{g, g^{-1}\}\subset\{e\}^{c}$, there exists an open neighborhood $W$ of $e$ such that $W\cap \{g, g^{-1}\}=\emptyset$, which is a contradiction.
\end{proof}

\begin{proposition}\label{p10000}
Let $G$ be a $T_{0}$-topological rough group. If $e\in G$ then $e\not\in\{g\}^{c}$ for each $g\in G\setminus\{e\}$.
\end{proposition}

\begin{proof}
Assume that there exists a point $g\in G\setminus\{e\}$ such that $e\in\{g\}^{c}$, then each open neighborhood of $e$ in $G$ contain the points $g$ and $g^{-1}$. Because $f: G\times G\rightarrow \overline{G}$ is continuous at the point $(e, g)$, each open neighborhood of $g$ must contain the rough identity element $e$, hence $g\in\{e\}^{c}$, that is, the set $\{e\}$ is not closed in $G$, which is a contradiction with Proposition~\ref{p0}.
\end{proof}

From Propositions~\ref{p0} and~\ref{p10000}, if $G$ is a finite topological rough group, then the set $\{e\}$ is open and closed in $G$, see Corollary~\ref{c0}.

\begin{corollary}\label{c0}
Let $G$ be a $T_{0}$-topological rough group. If $e\in G$ and $G$ is finite, then the set $\{e\}$ is open and closed in $G$.
\end{corollary}

The following example shows that a $T_{1}$-topological rough group $G$ with $e\not\in G$ need not to be Hausdorff.

\begin{example}\label{e3}
There exists a $T_{1}$-topological rough group $G$ such that $G$ is not Hausdorff and $e\not\in G$.
\end{example}

\begin{proof}
Let $U=\mathbb{Z}$ be a set of integer number and let ($\ast$) be the usual addition. A classification of $U$ is $U/R$, where $$E_{1}=\{4k+1, 4k+2: k\in\mathbb{Z}\}\ \mbox{and}\ E_{2}=\{4k+3, 4k+4: k\in\mathbb{Z}\}.$$ Let $G=\{2k+1: k\in\mathbb{Z}\}$. Then $\overline{G}=\mathbb{Z}$ and $0\not\in G$. Now we endow $\overline{G}$ with a topology $\tau$ as follows:

\smallskip
For each $x\in \mathbb{Z}\setminus G$, the point $x$ has a neighborhood base $\{\mathbb{Z}\}$;

\smallskip
For each $x\in G$, the point $x$ has a neighborhood base consisting of the sets of the form $(G\setminus F)\cup\{x\}$, where $F$ is an arbitrary finite subset of $\mathbb{Z}$.

Clearly, $$\tau_{G}=\{\emptyset\}\cup\{G\setminus F: F\ \mbox{is a finite subset of}\ G\}.$$ It is easy to check that $G$ is a $T_{1}$-topological rough group. However, $G$ is not Haudorff.
\end{proof}

The following theorem shows that $G$ is Hausdorff if $\{e\}$ is closed in $\overline{G}$.

\begin{theorem}\label{t4}
If $G$ is a topological rough group such that $\{e\}$ is closed in $\overline{G}$, then $G$ is Hausdorff.
\end{theorem}

\begin{proof}
Assume that $G$ is not Hausdorff, then there exist two distinct points $x, y\in G$ such that $x$ and $y$ cannot be separated by open neighborhoods in $G$. Clearly, $xy^{-1}\in\overline{G}$. We claim that $xy^{-1}\in\{e\}^{c}_{\tau}$ in $\overline{G}$, which will obtain a contradiction since $\{e\}$ is closed in $\overline{G}$. Indeed, take an arbitrary open neighborhood $W$ of $xy^{-1}$ in $\overline{G}$. Since $f: G\times G\rightarrow \overline{G}$ is continuous at point $(x, y^{-1})$, there exist open neighborhoods $U$ and $V$ of $x$ and $y^{-1}$ in $G$ respectively such that $UV\subset W$. Because the inverse mapping is a homeomorphism, $V^{-1}$ is an open neighborhood of $y$ in $G$. Since $U\cap V^{-1}\neq\emptyset$, there exists a point $z\in U\cap V^{-1}$, then $(z, z^{-1})\in U\times V$, hence $W$ contains the point $e$. Therefore, $xy^{-1}\in\{e\}^{c}_{\tau}$.
\end{proof}

By Theorem~\ref{t4}, we have the following corollary.

\begin{corollary}\label{c1}
If $G$ is a topological rough group with $\overline{G}$ being $T_{1}$, then $G$ is Hausdorff.
\end{corollary}

\begin{remark}
By Theorem~\ref{t3}, if $\overline{G}$ is $T_{0}$ then $G$ is $T_{1}$, but we cannot say that $\overline{G}$ is $T_{1}$, see Example~\ref{e0}; otherwise, $\overline{G}$ is Hausdorff by Corollary~\ref{c1}. Moreover, the following example is a complement of Example~\ref{e3}.
\end{remark}

\begin{example}\label{e4}
There exists a $T_{1}$-topological rough group $G$ such that $e\in G$ and $G$ is not Hausdorff, thus $G$ is not regular.
\end{example}

\begin{proof}
Take the same approximation space in Example~\ref{e3}. Let $$G=\{2k+1: k\in\mathbb{Z}\}\cup\{0\}.$$ Then $\overline{G}=\mathbb{Z}$ and $0\in G$. Now we endow $\overline{G}$ with a topology $\tau$ as follows:

\smallskip
For each $x\in \mathbb{Z}\setminus (G\setminus\{0\})$, the point $x$ has a neighborhood base $\{(\mathbb{Z}\setminus G)\cup\{0\}\}$;

\smallskip
For each $x\in G$, the point $x$ has a neighborhood base consisting of the sets of the form $(G\setminus F)\cup\{x\}$, where $F$ is an arbitrary finite subset of $\mathbb{Z}$.

Then $$\tau_{G}=\{\emptyset, \{0\}\}\cup\{G\setminus F: F\ \mbox{is a finite subset of}\ G\}.$$ It is easy to check that $G$ is a $T_{1}$-topological rough group. However, $G$ is not Haudorff.
\end{proof}

From Example~\ref{e4}, a $T_{1}$-topological rough group $G$ need not to be Hausdorff. However, we have the following result.

\begin{proposition}
Let $G$ be a $T_{1}$-topological rough group. If $e\in G$ then for each $g\in G\setminus\{e\}$ there exist open neighborhoods $U$ and $V$ of $e$ and $g$ in $G$ respectively such that $U\cap V=\emptyset$.
\end{proposition}

\begin{proof}
Suppose not, then there exists a point $g\in G\setminus\{e\}$ such that $U\cap V\neq\emptyset$ for any open neighborhoods $U$ and $V$ of $e$ and $g$ in $G$ respectively, where we may assume that $U$ is symmetric. Hence the intersection of any open neighborhood of the point $(e, g)$ in the product space $G\times G$ with the set $\{(z, z^{-1}): z\in G\}$ is non-empty. By the continuous of the mapping $f: G\times G\rightarrow \overline{G}$ at point $(e, g)$, each open neighborhood of $g$ in $G$ must contain the point $e$, hence $g\in\{e\}^{c}$. However, $\{e\}$ is closed in $G$ since $G$ is $T_{1}$, which is a contradiction.
\end{proof}

If $G$ is a $T_{0}$-strongly topological rough group, then the situation is quite different.

\begin{theorem}\label{t0}
If $G$ is a $T_{0}$-strongly topological rough group then $G$ is Hausdorff.
\end{theorem}

\begin{proof}
Assume that $G$ is not Hausdorff, then there exist two distinct points $x$ and $g$ in $G$ such that $U\cap V\neq\emptyset$ for any open neighborhoods $U$ and $V$ of $x$ and $g$ in $G$ respectively. We claim that $g\in \{x\}^{c}$ and $x\in \{g\}^{c}$, which is a contradiction with the $T_{0}$-property of $G$. Indeed, we prove that $g\in \{x\}^{c}$, and the proof of $x\in \{g\}^{c}$ is similar. Take any open neighborhood $W$ of $g$ in $\overline{G}$. Clearly, any open neighborhood of the point $(x, x^{-1}, g)$ in the subspace $G_{0}^{3}$ of the product space $(G\times G\times G)\cap G_{0}^{3}$ must contain some point of the form $(x, z^{-1}, z)$. Since $G$ is a strongly topological rough group, the mapping $f: G_{0}^{3}\rightarrow \overline{G}$ is continuous at the point $(x, x^{-1}, g)$, then $W$ contains the point $x$, thus $x\in W\cap G$. Therefore, $g\in \{x\}^{c}$. For the proof of $x\in \{g\}^{c}$, we only note that any open neighborhood of the point $(x, g^{-1}, g)$ in the subspace $G_{0}^{3}$ must contain some point of the form $(z, z^{-1}, g)$.
\end{proof}

However, the following two questions are open.

\begin{question}
Does there exist a Hausdorff topological rough group $G$ such that $G$ is not regular?
\end{question}

\begin{question}
If $G$ is strongly topological rough group $G$ with $\overline{G}$ being Hausdorff, is $G$ regular?
\end{question}

 \maketitle
\section{Basic properties of topological rough group}
In this section, we study some basic properties of topological rough groups. We mainly discuss the neighborhoods of rough identity element $e$ of topological rough groups.

Let $\tau$ be an arbitrary topology on $\overline{G}$. For any $g\in \overline{G}$, denote some neighborhoods base of $g$ in $\overline{G}$ by $\tau(g)$, and put $\tau_{G}(g)=\{U\cap G: U\in\tau(g)\}$. Then it follows from the definition of topological rough group, we have the following theorem, and the proof is obvious and we leave it to the reader.

\begin{theorem}\label{t6}
Let $G$ be a rough group and let $\tau$ be a topology on $\overline{G}$. For each $g\in G\cup G^{2}$, if we can choose a neighborhoods base $\tau(g)$ of $g$ in $\overline{G}$ satisfies the following conditions, then $G$ is a topological rough group.

\smallskip
(i) For any $g\in G$, $\tau_{G}(g)^{-1}=\tau_{G}(g^{-1})$;

 \smallskip
(ii) For any $g, h\in G$ and $W\in \tau(gh)$, there exist $U\in\tau_{G}(g)$ and $V\in\tau_{G}(h)$ such that $UV\subset W$.
\end{theorem}

In \cite{BIO2016}, the authors gave the following proposition.

\begin{proposition}\cite[Proposition 17]{BIO2016}
Let $G$ be a topological rough group and let $W\subset \overline{G}$ be an open subset with $e\in W$. Then there
exists an open set $V$ with $e\in V$ such that $V=V^{-1}$ and $V^{2}\subset W$.
\end{proposition}

However, we find that the proof of this result has a gap when $e\not\in G$. Indeed, we have the following counterexample.

\begin{example}
Let $U=\mathbb{Z}$ be the set of all integer number and let ($\ast$) be the usual addition. A classification of $U$ is $U/R$, where $$E_{1}=\{4k+1, 4k+2: k\in\mathbb{Z}\}\ \mbox{and}\ E_{2}=\{4k+3, 4k+4: k\in\mathbb{Z}\}.$$ Let $G=\{2k+1: k\in\mathbb{Z}\}$. Then $\overline{G}=\mathbb{Z}$ and $0\not\in G$. Now we endow $\overline{G}$ with a topology $\tau$ as follows:

\smallskip
For each $x\in \mathbb{Z}\setminus (G\cup\{0\})$, the point $x$ has a neighborhood base $\{\mathbb{Z}\}$;

\smallskip
For $x=0$, the point $0$ has a neighborhood base $\{(\mathbb{Z}\setminus (G\cup\{2\}))\cup\{4k+1: k\in\omega\}\}$;

\smallskip
For each $n\in\omega$, the point $2n+1$ has a neighborhood base $\{\{4k+2n+1: k\in\omega\}\}$, and the point $-2n-1$ has a neighborhood base $\{\{-4k-2n-1: k\in\omega\}\}$.

Clearly, $G$ is a topological rough group under the topology $\tau_{G}$. However, the neighborhood $V$ of $e$ in $\overline{G}$ is $(\mathbb{Z}\setminus (G\cup\{2\}))\cup\{4k+1: k\in\omega\}$ and $V\cap G=\{4k+1: k\in\omega\}$ are all not symmetric. Moreover, $2\in (V\cap G)^{2}$, but $2\not\in V$, hence $(V\cap G)^{2}\nsubseteq V$.
\end{example}

But we have the following proposition.

\begin{proposition}
Let $G$ be a topological rough group. If $e\in G$ and $U$ is an open neighborhood of $e$ in $\overline{G}$, then there
exists an open set $V$ in $G$ with $e\in V$ such that $V=V^{-1}$ and $V^{2}\subset U$.
\end{proposition}

\begin{proof}
Since $e\in G$ and $f: G\times G\rightarrow \overline{G}$ is continuous, $f^{-1}(U)$ is open in $G\times G$ and $(e, e)\in f^{-1}(U)$. Therefore, it follows from Theorem~\ref{t6} that there exists a symmetric open neighborhood $V$ of $e$ in $G$ such that $V^{2}\subset U$.
\end{proof}

Recall that a topological space $X$ is said to be {\it homogeneous} if for every $\in X$ and every
$y\in X$, there exists a homeomorphism $f$ of $X$ onto itself such that $f(x)=y$.
It is well known that each topological group is a homogeneous space. However, there exists a topological rough group which is not a homogeneous space, see  Example~\ref{e0}. In \cite{BIO2016}, the authors proved the following proposition.

\begin{proposition}\cite[Proposition 13]{BIO2016}\label{p2}
Let $G$ be a topological rough group and fix $a\in G$. Then

\smallskip
(a) The mapping $L_{a}: G\rightarrow \overline{G}$ defined by $L_{a}(x)=ax$ is one-to-one and continuous, for every $x\in G$;

\smallskip
(b) The mapping $R_{a}: G\rightarrow \overline{G}$ defined by $R_{a}(x)=xa$ is one-to-one and continuous, for every $x\in G$;

\smallskip
(c) The mapping $f: G\rightarrow G$ defined by $f(x)=x^{-1}$ is a homeomorphism, for every $x\in G$.
\end{proposition}

Therefore, it is natural to pose the following question.

\begin{question}\label{q0}
When is a topological rough group a homogeneous space?
\end{question}

Next we give some partial answers to Question~\ref{q0}.

\begin{proposition}\label{p3}
Let $G$ be a topological rough group such that $G$ is open in $\overline{G}$ and $e\in G$, and let $\mathscr{B}$ be a family of open neighborhood base of $e$ in $G$. For each $g\in G$, put $$\mathscr{L}_{g}=\{(aU)\cap G: U\in\mathscr{B}\}\ \mbox{and}\ \mathscr{R}_{g}=\{(Ua)\cap G: U\in\mathscr{B}\}.$$
Then $\mathscr{L}_{g}$ and $\mathscr{R}_{g}$ are two families of neighborhood bases of $g$ in $G$.
\end{proposition}

\begin{proof}
We prove that $\mathscr{L}_{g}$ is a family of neighborhood base of $g$ in $G$. The proof of $\mathscr{R}_{g}$ is similar.

Since the mapping $f: G\times G\rightarrow \overline{G}$ at point $(g, e)$ is continuous and $G$ is open in $\overline{G}$, it suffices to prove that each $gU\cap G$ is a neighborhood of $g$ in $G$ for each $U\in\mathscr{B}$. Take an arbitrary $U\in\mathscr{B}$. Then there exists $V\in\mathscr{B}$ such that $V\subset U$ and $gV\subset G$. We claim that $gV$ is open in $G$. Indeed, from Proposition~\ref{p2}, the mapping $L_{g^{-1}}$ is one-to-one continuous from $G$ to $\overline{G}$, then $L_{g^{-1}}^{-1}(V)$ is open in $G$ since $V$ is open in $G$ (thus in $\overline{G}$). Since $L_{g^{-1}}$ is an one-to-one mapping and $g V\subset G$, it follows that $L_{g^{-1}}^{-1}(V)=gV$, hence $gV$ is open in $G$. From $gV\subset gU$ and $gV\subset G$, it follows that $gU\cap G$ is a neighborhood of $g$ in $G$.
\end{proof}

\begin{proposition}\label{p8}
Let $G$ be a topological rough group such that $G$ is open in $\overline{G}$ and $e\in G$. For each $g\in G$ and $U\subset G$, if the interior of $gU\cap G$ is non-empty in $G$, then the interior of $U$ is also non-empty in $G$.
\end{proposition}

\begin{proof}
Let $w\in \mbox{Int}(gU\cap G)\subset G$. Then there exists a point $u\in U$ such that $w=gu$. Since the multiplication is continuous and $G$ is open in $\overline{G}$, there exists an open neighborhood $V$ of $u$ in $G$ such that $gV\subset \mbox{Int}(gU\cap G)$, hence $gV\subset gU$, which implies that $V\subset U$ since $V=g^{-1}gV\subset g^{-1}g U=U$. Hence the interior of $U$ is also non-empty in $G$.
\end{proof}

\begin{proposition}
Let $G$ be a topological rough group such that $G$ is open in $\overline{G}$ and $e\in G$. If $A$ is a subset of $G$ and $U$ an open neighborhood of $e$ in $G$. Then both the sets $AU\cap G$ and $UA\cap G$ are the neighborhoods of $A$ in $G$.
\end{proposition}

\begin{proof}
By Proposition~\ref{p3}, for each $g\in A$ there exists an open neighborhood $V_{g}$ of $e$ such that $gV_{g}$ is a neighborhood of $g$ and $gV_{g}\subset (G\cap gU)$. Hence $\bigcup_{g\in A}gV_{g}\subset AU\cap G$, that is, $AU\cap G$ a neighborhood of $A$ in $G$. A similar argument applies in the case $UA\cap G$.
\end{proof}

If $(G, \sigma)$ is a topological group and $A\subset G$ a subset of $G$, then the closure of $A$ in $G$ is precise the set $\bigcap\{AU: U\in\sigma(e)\}$, where $\sigma(e)$ is the set of all neighborhoods of $e$ in $G$, see \cite[Theorem 1.4.5]{AA}. For topological rough groups, we have the following proposition.

\begin{proposition}\label{p4}
Let $G$ be a topological rough group such that $G$ is open in $\overline{G}$ and $e\in G$. If $A$ is a subset of $G$, then $A^{c}=\bigcap\{AU: U\in\mathscr{B}\}$, where $\mathscr{B}$ is a family of open neighborhood base of $e$ in $G$ such that each element of $\mathscr{B}$ is symmetric.
\end{proposition}

\begin{proof}
We first prove that $A^{c}\subset AU$ for each $U\in \mathscr{U}$. Take an arbitrary $g\in A^{c}$. By Proposition~\ref{p3}, $gU\cap G$ is a neighborhood of $g$ in $G$, hence $gU\cap A\neq\emptyset$, then $g\in AU^{-1}=AU$. Thus $A^{c}\subset AU$. Next we prove that $\bigcap\{AU: U\in\mathscr{B}\}\subset A^{c}$.

Indeed, it suffices to verify that if $g$ is not in $A^{c}$ then there exists $U\in \mathscr{B}$ such that $g\not\in AU$. Since $g\not\in A^{c}$, it follows from Proposition~\ref{p3} that there exists $U\in\mathscr{B}$ such that $gU\cap A=\emptyset$, then $g\not\in AU^{-1}=AU$. Therefore, $A^{c}\subset AU$.
\end{proof}

\begin{remark}
(1) In Proposition~\ref{p4}, the conditions ``$G$ is an open subset'' cannot be omitted. Indeed, let $A=\{2, 6\}$ in Example~\ref{e2}. Then $A^{c}=\{2, 5, 6, 9\}\neq A$ and $\bigcap \{AU: U\ \mbox{is an open neighborhood of}\ e\ \mbox{in}\ G\}=A\{1\}=A$.

\smallskip
(2) If a topological rough group $G$ satisfies the conditions in Proposition~\ref{p3}, then $G$ is locally homeomorphic. Moreover, from the definition of topological rough group and Proposition~\ref{p3}, it is easy to see the following theorem. We leave the proof to the reader.
\end{remark}

\begin{theorem}\label{t5}
Let $G$ be a topological rough group such that $G$ is open in $\overline{G}$ and $e\in G$, and let $\mathscr{B}$ be a family of open neighborhood base of $e$ in $G$. Then:

\smallskip
(i) for each $U\in\mathscr{B}$, there is an element $V\in\mathscr{B}$ such that $V^{2}\subset U$;

\smallskip
(ii) for each $U\in\mathscr{B}$, there is an element $V\in\mathscr{B}$ such that $V^{-1}\subset U$;

\smallskip
(iii) for each $U\in\mathscr{B}$ and $x\in U$, there is an element $V\in\mathscr{B}$ such that $Vx\subset U$;

\smallskip
(iv) for each $U\in\mathscr{B}$ and $g\in G$, there is an element $V\in\mathscr{B}$ such that $gVg^{-1}\subset U$;

\smallskip
(v) for any $U, V\in\mathscr{B}$, there is an element $W\in\mathscr{B}$ such that $W\subset U\cap V$.
\end{theorem}

\begin{remark}
 (1) For any a group $U$ and any symmetric subset $G$ of $U$, if for an arbitrary topology $\tau$ on $U$ such that each point of $G\cup\{e\}$ is isolated point in $U$ and $G^{2}\subset \overline{G}$, then $G\cup \{e\}$ is a topological rough group and satisfies the conditions in Theorem~\ref{t5}; in particular, any finite topological rough groups $G$ with $\overline{G}$ being $T_{1}$ satisfy the conditions in Theorem~\ref{t5}.

(2) In Theorem~\ref{t5}, the condition ``$e\in G$'' cannot omit. Indeed, the topological rough group $G$ in Example~\ref{e1} is open in $\overline{G}$; however, there does not exists a neighborhood base which satisfy the conditions in Theorem~\ref{t5}.
\end{remark}

\begin{proposition}
Let $G$ be a topological rough group. Then we have the following properties:

\smallskip
(1) If $e\in G$, then for each open neighborhood $U$ of $e$ in $G$, there exists a symmetric open neighborhood $V$ of $e$ in $G$ such that $V^{2}\cap G\subset U$.

\smallskip
(2) If $\{e\}$ is open in $\overline{G}$, then $G$ is a discrete space.
\end{proposition}

\begin{proof}
(1) Take an arbitrary open neighborhood $U$ of $e$ in $G$. Then there exists an open neighborhood $W$ of $e$ in $\overline{G}$ such that $U=W\cap G$. Since $f: G\times G\rightarrow \overline{G}$ is continuous at point $(e, e)$ and the inverse mapping is a homeomorphism, there exists a symmetric open neighborhood $V$ of $e$ in $G$ such that $V^{2}\subset W$, hence $V^{2}\cap G\subset U$.

(2) It suffice to prove that $\{g\}$ is open in $G$ for any $g\in G$. Take an arbitrary $g\in G$. Since $f: G\times G\rightarrow \overline{G}$ is continuous at point $(g, g^{-1})$, there exists an open neighborhood $V$ of $g$ in $G$ such that $VV^{-1}\subset \{e\}$, then $VV^{-1}=\{e\}$ because $VV^{-1}\neq\emptyset$. By the uniqueness of the rough inverse element of each element of $G$, we have $V=\{g\}$. Therefore, $G$ is discrete.
\end{proof}

Finally, we discuss the extremally disconnected topological rough group. We recall that a space $X$ is {\it extremally disconnected} if the closure of any open subset of $X$ is open. The following theorem is important in our proof.

\begin{theorem}\cite{F1967}\label{t10}
Let $X$ be an extremally disconnected Hausdorff space,
and let $h$ be a homeomorphism of $X$ onto itself. Then the set $M=\{x\in X: h(x)=x\}$ of all fixed
points of $h$ is an open and closed subset of $X$.
\end{theorem}

\begin{theorem}\label{t8}
Let $G$ be a Hausdorff extremally disconnected topological rough group. If $G$ is open in $\overline{G}$ and $e\in G$, then there exists an open and closed Abelian neighborhood $O$ of $e$ in $G$ such that $a^{2}=e$, for each $a\in O$.
\end{theorem}

\begin{proof}
The mapping $h:G\rightarrow G$ defined by $h(g)=g^{-1}$ for each $g\in G$ is a
homeomorphism of $G$. It follows from Theorem~\ref{t10} that the set $U=\{g\in G: g^{2}=e\}$ is an open and closed  neighborhood of the rough identity element $e$. Since $e\in G$ and $G$ is open in $\overline{G}$, there exist symmetric open neighborhood $V$ and $W$ of $e$ in $G$ such that $W^{2}\subset V\subset V^{2}\subset U$. By Proposition~\ref{p4}, $W^{c}\subset W^{2}$. Since $G$ is extremally disconnected, $W^{c}$ is open in $G$. Clearly, $W^{c}$ is symmetric.

Next it suffices to prove that every two elements $a$ and $b$ of $V$ commute. Indeed, $abab=e$ since $ab\in U$. Now from $a^{2}=e$ and $b^{2}=e$ it follows that $ab=ba$.
Therefore, $V$ is Abelian, thus $W^{c}$ is Abelian. Now let $O=W^{c}$. It is just the requirement in our theorem.
\end{proof}

\begin{corollary}
Let $G$ be a Hausdorff extremally disconnected topological rough group. If $G$ is open in $\overline{G}$ and $e\in G$, then,
for every element $g$ of $G$ of order $2$, there exists an open neighborhood $V$ of the rough identity
element $e$ such that $g$ commutes with every element of $V \cup gV$.
\end{corollary}

\begin{proof}
By Theorem~\ref{t8}, the set $U$ of all elements of $G$ of order 2 is open in $G$. Take an arbitrary $g\in U$. Since $e\in G$ and $G$ is open in $\overline{G}$, there exists an open neighborhood $V$ of the rough identity
element $e$ such that $V\subset U$ and $gV\subset U$. Take any $a\in V$. Then $ga\in U$ and, therefore,
$gaga=e$. Since $g^{2}=e$ and $a^{2}=e$, it follows that $ga=ba$. Hence $g$ commutes with every
element of $V$.

Now, let $c\in gV$. Hence $c=gb$ for some $b\in V,$ and $$gc = ggb, cg= gbg=ggb.$$
Thus, $gc =cg$.
\end{proof}

\begin{theorem}
Let $G$ be a Hausdorff extremally disconnected topological rough group. For any $g\in G$, if $gG= Gg$, then the set $G_{g}=\{x\in G: xg=gx\}$ is an open and closed neighborhood of $g$.
\end{theorem}

\begin{proof}
Since $gG= Gg$, it is easy to see that the mapping $\phi(x)=g^{-1}xg$ is a homeomorphism of the space $G$ onto itself. Then it follows from Theorem~\ref{t10} that the set $F$ of all fixed points under $\phi$ is open and closed. It suffices to prove that $F=G_{g}$. Indeed, $\phi(x)=x$ if and only if $x=g^{-1}xg$ if and only if $gx=xg$.
\end{proof}

\begin{theorem}\label{t9}
Let $G$ be a Hausdorff extremally disconnected topological rough group. For any $g\in G$, if $gG= G$, then the set $M_{g}=\{x\in G: x^{2}=g\}$  is an open and closed set in $G$.
\end{theorem}

\begin{proof}
Since $gG= G$, it is easy to see that the mapping $\phi(x)=gx^{-1}$ is a homeomorphism of the space $G$ onto itself. Then it follows from Theorem~\ref{t10} that the set $F$ of all fixed points under $\phi$ is open and closed. Hence
$F$ coincides with $M_{g}$. Indeed, for any $x\in G$, $x\in G$ if and only if $\phi(x)=x$ if and only if $x=gx^{-1}$ if and only if
$x^{2}=g$ if and only if $x\in M_{g}$.
\end{proof}

\begin{remark} It should be noted that if $G$ is an extremally disconnected Hausdorff topological rough group with $e\in G$ and $G$ being open in $\overline{G}$, then the set $L=\{x\in G: x^{3}=e\}$ need not be open in $G$. Indeed, assume that $L$ is open, hence $L$ is a
neighborhood of $e$; then $M_{e}\cap L$ is also an open neighborhood of
$e$ in $G$ (see Theorem~\ref{t9}). However, it is clear that $L\cap M_{e}=\{e\}$; thus $e$
is isolated in $G$, which implies that $G$ is discrete by Proposition~\ref{p3}.
\end{remark}

  \maketitle
\section{Topological Rough subgroups of topological rough groups}
In this section we discuss some properties of topological rough subgroups of topological rough groups. First, we prove the following proposition.

\begin{proposition}
Let $G$ be a topological rough group. If $e\in G$ and $G$ is open in $\overline{G}$, then $H=\bigcap\{U: U\in\tau_{G}(e)\}$ is a topological group.
\end{proposition}

\begin{proof}
Since $e\in G$, it follows from Proposition~\ref{p2} that $H$ is symmetric. It suffices to prove that $H^{2}=H$.
Clearly, $H\subset H^{2}$. Take any $x, y\in H$. We prove that $xy\in H$. From $x, y\in H$ we have $x, y\in U$ for each $U\in\tau_{G}(e)$. For each $U\in\tau_{G}(e)$, since $G$ is a topological rough group and $G$ is open in $\overline{G}$, there exists $V\in\tau_{G}(e)$ such that $VV\subset U$, thus $xy\in VV\subset U$. Then $xy\in \bigcap\{U: U\in\tau_{G}(e)\}$. Therefore, $H^{2}\subset H$.
\end{proof}

In \cite{BIO2016}, the authors proved that the intersection of two topological rough subgroups $H_{1}$ and $H_{2}$ of the topological rough group $G$ is a topological rough subgroup if $\overline{H_{1}}\cap \overline{H_{2}}=\overline{H_{1}\cap H_{2}}$. However, the condition ``$\overline{H_{1}}\cap \overline{H_{2}}=\overline{H_{1}\cap H_{2}}$'' is only a sufficient condition but not a necessary, see Example~\ref{e2}. Indeed, put $H_{1}=\{2, 6\}$ and $H_{2}=\{5, 9\}$. Then $H_{1}, H_{2}$ and $H_{1}\cap H_{2}=\emptyset$ are all topological rough groups. But $\overline{H_{1}}\cap \overline{H_{2}}=\{1, 2, 5\}\neq\emptyset$. Therefore, it is natural to pose the following question.

\begin{question}
Let $G$ be a topological rough group and let $H_{1}, H_{2}$ be two rough subgroups. How to characterize the properties of $H_{1}, H_{2}$ such that $H_{1}\cap H_{2}$ is also a rough subgroup?
\end{question}

In \cite{AAAO2019}, the authors proved the following result.

\begin{theorem}\cite[Theorem 3.2]{AAAO2019}\label{t7}
Let $G$ be a topological rough group, and let $H$ be a rough subgroup. If $H_{\tau}^{c}$ in
$\overline{G}$ is a subset of $G$, then $H_{\tau}^{c}$ is a rough subgroup in $G$.
\end{theorem}

However, we also find that the proof this result has a gap. Indeed, we have the following counterexample.

\begin{example}
Let $U=\mathbb{Z}$ be a set of integer number and let ($\ast$) be the usual addition `+'. Let $G=\{2k+1: k\in\mathbb{Z}\}$ and $H=\{8k+1: k\in\omega\}\cup\{-8k-1: k\in\omega\}$. A classification of $U$ is $U/R$, where $E_{1}=\{8k+1: k\in\mathbb{N}\}$, $E_{2}=\{-8k-1: k\in\mathbb{N}\}$, $E_{3}=\{1, -1\}\cup (H+H)$, $E_{4}=\{6, -6, 5, 7\}$ and $E_{5}=\mathbb{Z}\setminus \bigcup_{i=1}^{4}E_{i}$. Clearly, $\overline{G}=\mathbb{Z}$, $\overline{H}=\bigcup_{i=1}^{3}E_{i}$ and $0\not\in G$. Now we endow $\overline{G}$ with a topology $\tau$ as follows:

\smallskip
For each $x\in \mathbb{Z}\setminus G$, the point $x$ has a neighborhood base $\{\mathbb{Z}\setminus G\}$;

\smallskip
For each $n\in H$, the point $n$ has a neighborhood base $\{H\}$;

\smallskip
For each $n\in \{3, -3\}$, the point $n$ has a neighborhood base $\{\{n\}\cup H\}$;

\smallskip
For each $n\in G\setminus (H\cup\{3, -3\})$, the point $n$ has a neighborhood base $\{\{n\}\}$.

Clearly, $G$ is a topological rough group under the topology $\tau_{G}$, and $H$ a topological rough subgroup. Clearly, the closure of $H$ in $\overline{G}$ is the set $H\cup\{3, -3\}$, which is contained in $G$. However, $3\in H\cup\{3, -3\}$ and $6\not\in \overline{H\cup\{3, -3\}}=G\setminus\{6, -6, 5, 7\}$. Therefore, $H_{\tau}^{c}=H\cup\{3, -3\}$ is not a topological rough subgroup in $G$.
\end{example}

Therefore, it is interesting to discuss the following question.

\begin{question}
Let $G$ be a topological rough group and let $H$ be a rough subgroup of $G$. When is $H^{c}_{\tau}$ a topological rough subgroup?
\end{question}

In general, an open topological rough subgroup is not closed in a topological rough group, see \cite{BIO2016}. Hence it is natural to consider the following question.

\begin{question}\label{q1}
Let $H$ be an open topological rough subgroup of a topological rough subgroup of $G$. When is $H$ closed in $G$?
\end{question}

Next we discuss Question~\ref{q1} as follows.

\begin{definition}\cite{BO2016}
Let $U$ be a topological group. Let $N$ be a normal subgroup
of $U$ and let $G$ be a nonempty subset of $U$. A classification of $U$ is $U/R=\{uN: u\in U\}$ (that is, the set of coset space). If $G$ is a topological rough group, then we say that $G$ is a {\it topological rough group with respect to the normal subgroup $N$}.
\end{definition}

\begin{proposition}\label{p5}
Let $U$ be a topological group, $N$ an open normal subgroup of $U$ and $G$ a topological rough group with respect to $N$.  If $H$ is a topological rough subgroup of $G$ with respect to $N$, then the following statements hold.

\smallskip
(1) $\overline{H_{\tau}^{c}}=\overline{H_{\tau}^{c}\cap G}=\overline{H}=HN$.

\smallskip
(2) The closure of $H$ in $\overline{G}$, $H_{\tau}^{c}$, is a topological rough group with respect to $N$, and $\overline{H_{\tau}^{c}}=\overline{H}=HN$.

\smallskip
(3) The closure of $H$ in $G$, $H^{c}$, is a topological rough subgroup of $G$ with respect to $N$, and $\overline{H^{c}}=\overline{H}=HN$.
\end{proposition}

\begin{proof}
(1) Clearly, $$HN=\overline{H}\subset\overline{H_{\tau}^{c}\cap G}\subset\overline{H_{\tau}^{c}}.$$ From \cite[Theorem 1.4.5]{AA}, it follows that $$H_{\tau}^{c}=\overline{G}\cap\bigcap\{HW: W\ \mbox{is any open neighborhood of}\ e\ \mbox{in}\ U\}$$ and $\overline{G}=GN$,
then $H_{\tau}^{c}\subset GN\cap HN=HN=\overline{H}$. Hence $$\overline{H_{\tau}^{c}}=H_{\tau}^{c}N\subset HNN=HN=\overline{H}.$$ Thus $\overline{H_{\tau}^{c}}=\overline{H_{\tau}^{c}\cap G}=\overline{H}=HN$.

(2) Since $H$ is a topological rough subgroup, $H$ is symmetric and $H^{2}\subset \overline{H}=HN$. Since the inverse mapping is a homeomorphism and $GN$ a symmetric set, $H_{\tau}^{c}$ is symmetric. Next it suffices to prove that $(H_{\tau}^{c})^{2}\subset \overline{H_{\tau}^{c}}$.

Take any $x, y\in (H_{\tau}^{c})^{2}$. Then $x, y\in HN$ since $(H_{\tau}^{c})^{2}\subset \overline{H_{\tau}^{c}}$, Hence $$xy\in HNHN=H^{2}N\subset \overline{H}N=HNN=HN\subset H_{\tau}^{c}N=\overline{H_{\tau}^{c}}.$$Therefore, $H_{\tau}^{c}$ is a topological rough subgroup with respect to $N$.

(3) By (2), $H^{c}$ is symmetric. It suffices to prove that $(H^{c})^{2}\subset \overline{H^{c}}$. Take any $x, y\in (H^{c})^{2}$. By the proof of (2), we have $xy\in\overline{H}$, thus $xy\in\overline{H^{c}}$ by (1).
\end{proof}

Let $G$ be a topological rough group and let $H$ be a symmetric subset of $G$. We say that $H$ is {\it $n$-order} if $$n=\min\{m\geq 2: H^{m}=H, m\in\mathbb{N}\},$$ where $n\in (\mathbb{N}\setminus\{1\})\cup\{\infty\}$ Obviously, $H$ is a group if and only if $H$ is 2-order, and a 3-order $H$ subset of $G$ is a group if and only if $e\in H$. The set of all odd integers is a 3-order, but not 2-order.

\begin{lemma}\label{l0}
Let $G$ be a topological rough group and $H$ be a 3-order subset of $G$. Then, for each $g, h\in \overline{G}$, we have $gH=hH$ or $gH\cap hH=\emptyset$.
\end{lemma}

\begin{proof}
Take any $g, h\in \overline{G}$. Assume that $gH\cap hH\neq\emptyset$. We prove that $gH=hH$. Indeed, take an arbitrary $x\in gH$. Then there exists $h_{1}\in H$ such that $x=gh_{1}$. Since $gH\cap hH\neq\emptyset$, it follows that there exist $h_{2}, h_{3}\in H$ such that $gh_{2}=hh_{3}$, then $g=hh_{3}h_{2}^{-1}$. Hence $x=(hh_{3}h_{2}^{-1})h_{1}=h(h_{3}h_{2}^{-1}h_{1})\in hH$ since $H$ is 3-order. Thus $gH\subset hH$. Similarly, we have $hH\subset gH$. Therefore, $gH=hH$.
\end{proof}

\begin{theorem}
Let $G$ be a topological group, and let $H$ be an open topological rough subgroup of $G$ with respect to an open normal subgroup $N$ of $G$. If $H$ is 3-order, then $H$ is closed in $G$.
\end{theorem}

\begin{proof}
By Proposition~\ref{p5}, $H^{c}$ is a topological rough subgroup of $G$. Since $H$ is open in $G$ and $H^{c}\subset G$, it follows that $H$ is open in $H^{c}$. For each $g\in G$, $L_{g}$ is a topological homeomorphism. By left transformation $L_{g}(H)=gH$ is open in $G$ since $H$ is open in $G$. Fix an arbitrary $h\in H$. For any $g\in H^{c}$, since $g\in L_{g}\circ L_{h}(H)$ and $L_{g}\circ L_{h}(H)$ is open in $G$, $H^{c}\cap L_{g}\circ L_{h}(H)$ is open in $H^{c}$. Therefore, it follow from $H=H^{3}$ and Lemma~\ref{l0} that the family $\{H^{c}\cap L_{g}\circ L_{h}(H): g\in H^{c}\}$ consists of disjoint open subsets and covers the set $H^{c}$. Hence, each element of the family $\{H^{c}\cap L_{g}\circ L_{h}(H): g\in H^{c}\}$ is also closed, then $H$ is closed in $H^{c}$. Since $H^{c}$ is closed in $G$, the set $H$ is closed in $G$.
\end{proof}

\begin{proposition}\label{p7}
Let $G$ be a topological rough group with $G$ open in $\overline{G}$ and $e\in G$, and let $H$ be a discrete normal rough subgroup of $G$. Then for each $x\in H$ there exists an open symmetric neighborhood $V$ of $e$ in $G$ such that $xy=yx$ for any $y\in V$.
\end{proposition}

\begin{proof}
If $H=\{e\}$, there is nothing to prove. Hence suppose that the subgroup $H\neq\{e\}$. Fix an arbitrary element $x\in H\setminus\{e\}$. Since the rough subgroup $K$ is discrete, there exists an open neighborhood $U$ of $x$ in $G$ such that $U\cap H=\{x\}$. Clearly, $x\in U$ and $U$ is open in $\overline{G}$ since $G$ is open in $\overline{G}$. By the continuity of the mapping $f: G\times G\rightarrow \overline{G}$, there exist open neighborhoods $V_{1}, W$ of $e$ and $x$ in $G$ respectively such that $WV_{1}\subset U$. Similarly, there exists an open symmetric neighborhood $V\subset V_{1}$ of $e$ in $G$ such that $Vx\subset W$. Then we have $(ex)e=x\in VxV\subset WV_{1}\subset U$. Let $y\in V$ be arbitrary. Then $yxy^{-1}\in H$ since $H$ is normal. Moreover, it is obvious that $$yxy^{-1}\in VxV\subset U.$$ Thus $yxy^{-1}\in H\cap U=\{x\}$, that is, $yxy^{-1}=x$.
\end{proof}

\begin{remark}
 From the proof above, this proposition holds when $G$ is a topological group.
\end{remark}

Finally, we discuss the connectedness of topological rough groups in this section. Let $G$ be a space. For each point $g\in G$, the {\it connected component} or, simply, {\it the component of $G$ at point $g$} is the union of all connected subsets of $G$ containing $g$, denote it by $C_{G}(g)$. Since
the union of any family of connected subspaces containing a given point is connected, it follow from the definition of topological rough group and the connectedness being preserved by continuous mapping that it easily to verify the following proposition, we leave the proof to the reader.

\begin{proposition}\label{p6}
Let $G$ be a topological rough group. Then we have the following statements:

\smallskip
(1) $C_{G}(g)^{-1}=C_{G}(g^{-1})$ for each $g\in G$;

\smallskip
(2) $\bigcup_{g\in G}C_{G}(g)C_{G}(g)^{-1}\subset C_{\overline{G}}(e)$;

\smallskip
(3) For each $g\in G$, $C_{G}(g)$ is closed in $G$;

\smallskip
(4) If $e\in G$, then $C_{G}(e)^{2}\subset C_{\overline{G}}(e)$;
\end{proposition}

\begin{proposition}
Let $(U, R)$ be an approximation space endowed with a topology $\theta$ such that the classification of $U$ is $U/R=\{C_{U}(s): s\in U\}$. If $G$ is a topological rough group with $\theta|_{\overline{G}}\subset\tau$ and $e\in G$, then $C_{G}(e)$ is a closed topological rough group.
\end{proposition}

\begin{proof}
By Proposition~\ref{p6}, $C_{G}(e)$ is symmetric and closed. It suffices to prove that $C_{G}(e)C_{G}(e)\subset \overline{C_{G}(e)}.$ Clearly, $\overline{C_{G}(e)}=C_{S}(e).$ Since $G$ is a topological rough group and the connectedness is preserved by continuous mapping, the subspace $C_{G}(e)C_{G}(e)$ is connected in $\overline{G}$, then $C_{G}(e)C_{G}(e)\subset C_{\overline{G}}(e)=C_{S}(e)$ since $e\in C_{G}(e)C_{G}(e)$ and $C_{S}(e)= C_{\overline{G}}(e)$. Hence $C_{G}(e)$ is a topological rough subgroup.
\end{proof}

\begin{proposition}
Let $(U, R)$ be an approximation space endowed with a topology $\theta$ such that the classification of $U$ is $U/R=\{C_{U}(s): s\in U\}$. If $G$ is a connected open topological rough group in $\overline{G}$ with $\theta|_{\overline{G}}\subset\tau$ and $e\in G$, then, for each neighborhood $U$ of $e$ in $G$, there exists an open neighborhood $V$ of $e$ in $G$ such that $V$ is a topological rough group.
\end{proposition}

\begin{proof}
Let $U$ be a neighborhood of $e$ in $G$. From the continuity of the multiplication at point $(e, e)$, there exists an open symmetric neighborhood $V$ of $e$ such that $V^{2}\subset U$. Then $V$ is a topological rough subgroup. Indeed, since $G$ is connected and $V^{2}\subset U\subset G$, it follows that $\overline{V}\subset \overline{G}=C_{S}(e)\subset\overline{V}$.
\end{proof}

\begin{proposition}
Let $G$ be a connected topological group, and let $H$ be a discrete normal rough subgroup of $G$ such that $e\in H$. Then $H$ is contained in the center of the group $G$, that is, every element of $H$ commutes with every element of $G$.
\end{proposition}

\begin{proof}
Take an arbitrary $x\in H$. From Proposition~\ref{p7} and the remark after the proof, there exists an open symmetric neighborhood $V$ of $e$ in $G$ such that $xy=yx$ for each $y\in V$. Obviously, $\bigcup_{n\in\mathbb{N}}V^{n}$ is an open subgroup, hence it is closed in $G$, then $G=\bigcup_{n\in\mathbb{N}}V^{n}$ since $G$ is connected. Therefore, every element $g\in G$ can be written in the form $g=y_{1}\cdots y_{n}$, where $n\in\mathbb{N}$ and $\{y_{1}, \cdots, y_{n}\}\subset V$. Because $x$ commutes with every element of $V$, it follows that $$gx=y_{1}\cdots y_{n}x=y_{1}\cdots xy_{n}=\cdots=y_{1}x\cdots y_{n}=x y_{1}\cdots y_{n}=xg.$$ Therefore, we have proved that the element $x\in H$ is in the center of the group $G$. By the arbitrary taking of the element of $H$, we conclude that the center of $G$ contains $H$.
\end{proof}

\maketitle
\section{Rough homomorphism on rough groups and open mapping theorem}
In this section we prove the version of the open mapping theorem in the class of topological rough groups. First, we redefine the concept of rough homomorphism. Indeed,  the authors in \cite{N2014} define the concept of the rough homomorphism, but this definition has some defect, for example, this definition does not consider any information of the upper approximation. We revise the definition as follows.

Let $(U_{1}, R_{1})$, $(U_{2}, R_{2})$ be two approximation spaces, and let $\ast_{1}$, $\ast_{2}$ be two binary operations over universes $U_{1}$ and $U_{2}$, respectively. Let $G_{1}$ and $G_{2}$ be topological rough groups. In this section, we always denote $e_{1}$ and $e_{2}$ are the rough identity elements of $G_{1}$ and $G_{2}$ respectively

\begin{definition}
Let $G_{1}\subset U_{1}$ and $G_{2}\subset U_{2}$ be rough groups. We say that $G_{1}$ and $G_{2}$ are {\it rough homomorphism} if there exists a surjection mapping $\varphi: \overline{G_{1}}\rightarrow \overline{G_{2}}$ such that the following conditions (1)-(3) hold:

\smallskip
(1) $\varphi|_{G_{1}}$ is a surjection mapping from $G_{1}$ to $G_{2}$;

\smallskip
(2) For any $x, y\in G_{1}\cup\{e\}$, we have $\varphi(x\ast _{1}y)=\varphi(x)\ast_{2} \varphi(y);$

\smallskip
(3) For any subset $H$ of $G_{1}$, $\overline{H}=\varphi^{-1}(\overline{\varphi(H)})$.

\smallskip
If a rough homomorphism is a bijection, then we say that $G_{1}$ and $G_{2}$ are {\it rough isomorphism}.
\end{definition}

First of all, we give some properties of rough homomorphism.

\begin{proposition}\label{p1}
Let $G_{1}\subset U_{1}$ and $G_{2}\subset U_{2}$ be rough groups that are rough homomorphism. If $\varphi: \overline{G_{1}}\rightarrow \overline{G_{2}}$ is a rough homomorphism, then $\varphi(e_{1})=e_{2}$ and $\varphi(g^{-1})=\varphi(g)^{-1}$ for any $g\in G_{1}$.
\end{proposition}

\begin{proof}
We prove that $\varphi(e_{1})=e_{2}$. Indeed, take an arbitrary $y\in G_{2}$, since $\varphi|_{G_{1}}$ is a surjection mapping, there exists a $x\in G_{1}$ such that $\varphi(x)=y$. Then it follows from (2) of the definition of rough homomorphism that $$y=\varphi(x)=\varphi(x\ast _{1}e_{1})=\varphi(x)\ast_{2} \varphi(e_{1})=y\ast_{2} \varphi(e_{1}).$$ By the arbitrary taking of $y$, it follows that $\varphi(e_{1})=e_{2}$.

For any $g\in G_{1}$, since $\varphi(e_{1})=\varphi(g^{-1}\ast_{1} g)=\varphi(g^{-1})\ast_{2} \varphi(g)=e_{2}$ and $\varphi(G_{1})=G_{2}$, we have $\varphi(g^{-1})=\varphi(g)^{-1}$.
\end{proof}

\begin{definition}
Let $G_{1}\subset U_{1}$ and $G_{2}\subset U_{2}$ be rough groups that are rough homomorphism. Then the set $\{x| \varphi(x)=e_{2}, x\in G_{1}\}$ is called {\it rough homomorphism kernel}, denoted by $\mbox{ker}(\varphi)$, where $\varphi$ is the rough homomorphism.
\end{definition}

\begin{remark}
 From the definition above, it maybe the rough homomorphism kernel is empty. Indeed, if a topological rough group does not contain the rough identity element, then the rough kernel of any identity mapping from $\overline{G}$ onto itself is empty.
 \end{remark}

Let $G$ be a topological rough group and let $H$ be a rough subgroup of $G$. We say that $H$ is {\it weakly rough normal} if $(g\ast H\ast g^{-1})\cap G\subset H$ for each $g\in G$.

\begin{proposition}
Let $\varphi$ be a rough homomorphism between rough groups $G_{1}$ and $G_{2}$. Then rough homomorphism kernel $\mbox{ker}(\varphi)$ is a weakly rough normal rough subgroup of $G_{1}$.
\end{proposition}

\begin{proof}
Let $N=\mbox{ker}(\varphi)$. We first prove that $N$ is a rough subgroup of $G_{1}$. Obviously, $N$ is symmetric by Proposition~\ref{p1}. Now we prove that $N\ast_{1} N\subset \overline{N}$. Take any $x, y\in N$. Then $\varphi(x\ast_{1} y)=\varphi(x)\ast_{2} \varphi(y)=e_{2}\ast_{2}e_{2}=e_{2}\in\overline{\{e_{2}\}}$. Hence $x\ast_{1} y\in \varphi^{-1}(\overline{\{e_{2}\}})=\overline{N}$. Thus, $N\ast_{1} N\subset \overline{N}$. Then $N$ is a rough group.

Next we prove that $N$ is weakly rough normal. Take any $g\in G_{1}$. Pick an arbitrary $x\in (g\ast_{1} N\ast_{1} g^{-1})\cap G_{1}$. We prove that $x\in N$. Indeed, since $x\in (g\ast_{1} N\ast_{1} g^{-1})\cap G_{1}$, there exists $n\in N$ such that $x=g\ast_{1} n\ast_{1} g^{-1}$. Then $x\ast_{1} g=x\ast_{1} n$, thus $\varphi(x)\ast_{2} \varphi(g)=\varphi(x\ast_{1} g)=\varphi(g\ast_{1} n)=\varphi(g)\ast_{2} \varphi(n)=\varphi(g)\ast_{2} e_{2}=\varphi(g)$,  which implies $\varphi(x)=e$. Therefore, $x\in N$ since $x\in G_{1}$.
\end{proof}

\begin{proposition}
Let $\varphi$ be a rough homomorphism between rough groups $G_{1}$ and $G_{2}$. Then $\varphi((\mbox{ker}(\varphi))^{2})=\{e_{2}\}$.
\end{proposition}

\begin{proof}
Take any $x, y\in \mbox{ker}(\varphi)$. Since $\varphi$ is a rough homomorphism, $\varphi(xy)=\varphi(x)\varphi(y)=e_{2}e_{2}=e_{2}.$
\end{proof}

\begin{proposition}
Let $\varphi$ be a rough homomorphism between rough groups $G_{1}$ and $G_{2}$. If $H$ is a (normal) rough subgroup of $G_{1}$, then $\varphi(H)$ is a (normal) rough subgroup of $G_{2}$.
\end{proposition}

\begin{proof}
Assume that $H$ is a rough subgroup of $G_{1}$. Then it is obvious that $\varphi(H)$ is symmetric by Proposition~\ref{p1}. We need to prove that $\varphi(H)\ast_{2} \varphi(H)\subset\overline{\varphi(H)}$, that is, $\varphi(H\ast_{1} H)\subset\overline{\varphi(H)}$. Indeed, since $H$ is rough subgroup of $G_{1}$, it follows that $H\ast_{1} H\subset \overline{H}=\varphi^{-1}(\overline{\varphi(H)})$, thus $\varphi(H\ast_{1} H)\subset\overline{\varphi(H)}$.

If $H$ is normal, then $g\ast_{1} H=H\ast_{1} g$ for each $g\in G_{1}$. Take any $h\in G_{2}$. Since $\varphi|_{G_{1}}$ is onto, there exists $g\in G_{1}$ such that $\varphi(g)=h$. Then $$h\ast_{2} \varphi(H)=\varphi(g)\ast_{2} \varphi(H)=\varphi(g\ast_{2}H)=\varphi(H\ast_{2}g)=\varphi(H)\ast_{2}\varphi(g)=\varphi(H)\ast_{2}h.$$Hence $H$ is normal.
\end{proof}

\begin{proposition}
Let $\varphi$ be a rough homomorphism between rough groups $G_{1}$ and $G_{2}$. If $H$ is a (weakly rough normal) rough subgroup of $G_{2}$, then $\varphi^{-1}(H)\cap G_{1}$ is a (weakly rough normal) rough subgroup of $G_{1}$.
\end{proposition}

\begin{proof}
Assume that $H$ is a rough subgroup of $G_{2}$, then it is obvious that $\varphi^{-1}(H)\cap G_{1}$ is symmetric. Next we prove that $(\varphi^{-1}(H)\cap G_{1})\ast_{1} (\varphi^{-1}(H)\cap G_{1})\subset \overline{\varphi^{-1}(H)\cap G_{1}}$. Clearly, $\varphi((\varphi^{-1}(H)\cap G_{1})\ast_{1} (\varphi^{-1}(H)\cap G_{1}))=H\ast_{2} H\subset \overline{H}$, hence $(\varphi^{-1}(H)\cap G_{1})\ast_{1} (\varphi^{-1}(H)\cap G_{1})\subset\varphi^{-1}(\overline{H})$. Moreover, $\overline{\varphi^{-1}(H)\cap G_{1}}=\varphi^{-1}(\overline{\varphi(\varphi^{-1}(H)\cap G_{1})})=\varphi^{-1}(\overline{H})$. Therefore, $(\varphi^{-1}(H)\cap G_{1})\ast_{1} (\varphi^{-1}(H)\cap G_{1})\subset \overline{\varphi^{-1}(H)\cap G_{1}}$. Thus $\varphi^{-1}(H)\cap G_{1}$ is a rough subgroup of $G_{1}$.

If $H$ is weakly rough normal in $G_{2}$, then we claim that $\varphi^{-1}(H)\cap G_{1}$ is weakly rough normal in $G_{1}$. Indeed, take any $g\in G_{1}$. We prove that $(g\ast_{1} (\varphi^{-1}(H)\cap G_{1})g^{-1})\cap G_{1}\subset\varphi^{-1}(H)\cap G_{1}$. Let $x\in (g\ast_{1} (\varphi^{-1}(H)\cap G_{1})g^{-1})\cap G_{1}$. Then there exists $g_{1}\in\varphi^{-1}(H)\cap G_{1}$ such that $x=g\ast_{1} g_{1}\ast g^{-1}\in G_{1}$. Then $x\ast_{1} g=g\ast_{1} g_{1}$. Hence $\varphi(x)\ast_{2} \varphi(g)=\varphi(x\ast_{1} g)=\varphi(g\ast_{1} g_{1})=\varphi(g)\ast_{2} \varphi(g_{1})$, then $\varphi(x)=\varphi(g)\ast_{2} \varphi(g_{1})\ast_{2}\varphi(g)^{-1}\in G_{2}$ since $x\in G_{1}$. Hence $\varphi(g)\ast_{2} \varphi(g_{1})\ast_{2}\varphi(g)^{-1}\in (\varphi(g)\ast_{2} H\ast_{2}\varphi(g)^{-1})\cap G_{2}\subset H$, then $x\in \varphi^{-1}(H)\cap G_{1}$.
\end{proof}

\begin{proposition}
Let $\varphi$ be a rough homomorphism between rough groups $G_{1}$ and $G_{2}$. If $G_{1}$ is a topological group, then $G_{2}$ is also a topological group.
\end{proposition}

\begin{proof}
It suffices to prove that $G_{2}^{2}=G_{2}$. Take arbitrary $x, y\in G_{2}$. Then there exists $g, h\in G_{1}$ such that $\varphi(g)=x, \varphi(h)=y$, hence $\varphi(gh)=\varphi(g)\varphi(h)=xy$. Since $G_{1}$ is a topological group, it follows that $gh\in G_{1}$, thus $\varphi(gh)\in G_{2}$ by the definition of $\varphi$, that is, $xy\in G_{2}$. Therefore, $G_{2}$ is a topological group.
\end{proof}

Finally we prove the version of open mapping theorem for topological rough groups. We recall some concepts.

Let $X$ be a space. We say that $X$ is {\it locally compact} if for each point $x$ of $X$ there exists an open neighborhood $V$ of $x$ such that $V^{c}$ is a compact set. $X$ is a {\it $\sigma$-compact space} if it is the union of countably many compact subsets of $X$.

\begin{theorem}
Let $G_{1}$ and $G_{2}$ be locally compact Hausdorff topological rough groups and let $f$ be a
continuous rough homomorphism of $G_{1}$ onto $G_{2}$. If $G_{1}$ and $G_{2}$ satisfy the following conditions, then $f$ is open.

\smallskip
(1) $e_{1}\in G_{1}$ and $e_{2}\in G_{2}$;

\smallskip
(2) $G_{1}$ and $G_{2}$ are open in $\overline{G_{1}}$ and $\overline{G_{2}}$ respectively;

\smallskip
(3) $G_{1}$ is $\sigma$-compact.
\end{theorem}

\begin{proof}
By Proposition~\ref{p3} and the assumptions of (1) and (2), it suffices to prove that $f(U)$ is a neighborhood of $e_{2}$ in $G_{2}$ for each open symmetric neighborhood $U$ of $e_{1}$ in $G_{1}$.  Let $U$ be an open neighborhood of $e_{1}$ in $G_{1}$. Form Proposition~\ref{p4} and local compactness of $G_{1}$, there exists an open symmetric neighborhood $V$ of $e_{1}$ in $G_{1}$ such that $V^{c}\ast_{1} V^{c}\subset U$ and $V^{c}$ is compact. Clearly, $G_{1}=G_{1}\cap\bigcup_{g\in G_{1}} g\ast_{1} V$ and $V^{c}$ is symmetric. Since $G_{1}$ is $\sigma$-compact, there exists a countable subset $\{g_{n}: n\in\mathbb{N}\}\subset G_{1}$ such that $G_{1}=G_{1}\cap\bigcup_{n\in\mathbb{N}} g_{n}\ast_{1} V$. For each $n\in\mathbb{N}$, let $h_{n}=f(g_{n})$. Then $$G_{2}=G_{2}\cap\bigcup_{n\in\mathbb{N}} f(g_{n}\ast_{1} V^{c})=G_{2}\cap\bigcup_{n\in\mathbb{N}}f(g_{n})\ast_{2} f(V^{c})=G_{2}\cap\bigcup_{n\in\mathbb{N}}h_{n}\ast_{2} f(V^{c})$$ because $f$ is a surjective rough homomorphism. Moreover, $f(V^{c})$ is compact in $G_{2}$ since $V^{c}$ is compact, then it follows that each $h_{n}\ast_{2} f(V^{c})$ is compact in $\overline{G_{2}}$ since the multiplication of $G_{2}$ is continuous, hence each $h_{n}\ast_{2} f(V^{c})$ is closed in $\overline{G_{2}}$, thus each $(h_{n}\ast_{2} f(V^{c}))\cap G_{2}$ in closed in $G_{2}$. By the local compactness of $G_{2}$, there exists $n\in\mathbb{N}$ such that the interior of $(h_{n}\ast_{2} f(V^{c}))\cap G_{2}$ is non-empty in $G_{2}$. By Proposition~\ref{p8}, the interior of $f(V^{c})$ in $G_{2}$ is non-empty. So there exists a non-empty open subset $W$ of $G_{2}$ such that $W\subset f(V^{c})$. Take an arbitrary $w\in W$. Then there exists a point $v\in V^{c}$ such that $f(v)=w$. Therefore, $$e_{2}\in w^{-1}\ast_{2}W\subset w^{-1}\ast_{2}f(V^{c})=f(v^{-1})\ast_{2}f(V^{c})\subset f(v\ast_{1}V^{c})\subset f(V^{c}\ast_{1}V^{c})\subset f(U),$$thus $f(U)$ is a neighborhood of $e_{2}$ in $G_{2}$.
\end{proof}


\begin{thebibliography}{99}
\bibitem{AAAO2019} N. Alharbi, A. Altassan, H. Aydi, C. \"{O}zel, {\it On topological rough groups}, arxiv.org/abs/1909.02500v1.

\bibitem{AA} A.V. Arhangel' ski\v\i, M. Tkachenko, {\it Topological Groups and Related Structures}, Atlantis Press and World Sci., 2008.

\bibitem{BIO2016} N. Ba\v{g}{\i}rmaz, \`{I}. I\c{c}en, A.F. \"{O}zcan, {\it Topological rough groups}, Topol. Algebra Appl., {\bf 4}(2016), 31--38.

\bibitem{BO2016} N. Ba\v{g}{\i}rmaz, A.F. \"{O}zcan, \`{I}. I\c{c}en, {\it Rough approximations in a topological group}, Gen. Math. Notes, {\bf 36}(2016), 1--18.

\bibitem{Bi1994} R. Biswas, S. Nanda, {\it Rough groups and rough subgroups}, Bull. Polish Acad. Sci. Math., 42(1994), 251--254.

\bibitem{BO1995} Z. Bonikowaski, {\it Algebraic structures of rough sets}, in: W.P. Ziarko (Ed.), Rough Sets, Fuzzy Sets and Knowledge Discovery,
Springer-Verlag, Berlin, 1995, pp. 242每247.

\bibitem{CMN2007} W. Cheng, Z. Mo, J. Wang, {\it Notes on ``the lower and upper approximations in a fuzzy group'' and ``rough ideals in semigroups''},
Inf. Sci., 177(2007), 5134--5140.

\bibitem{Da2004} B. Davvaz, {\it Roughness in rings}, Inf. Sci., 164(2004), 147--163.

\bibitem{DC2018} L. D$^{\prime}$eer, C. Cornelis, {\it A comprehensive study of fuzzy covering-based rough set models: Definitions, properties and interrelationships}, Fuzzy Sets and Systems, 336(2018), 1--26.

 \bibitem{E1989} R. Engelking, {\it General Topology} (revised and completed edition), Heldermann Verlag, Berlin, 1989.

\bibitem{F1967} Z. Frol\'{\i}k, {\it Homogeneity problems for extremally disconnected spaces}, Comment. Math. Univ.
Carolin, 8(1967), 757--763.

\bibitem{I1987} T. Iwinski, {\it Algebraic approach to rough sets}, Bull. Polish Acad. Sci. Math., 35(1987), 673--683.

\bibitem{KM2006} M. Kondo, {\it On the structure of generalized rough sets}, Inf. Sci., 176(2006), 589--600.

\bibitem{K1997} N. Kuroki, {\it Rough ideals in semigroups}, Inf. Sci., 100(1997), 139-每163.

\bibitem{KW1996} N. Kuroki, P.P. Wang, {\it The lower and upper approximations in a fuzzy group}, Inf. Sci., 90(1996), 203--220.

\bibitem{LZ2014} F. Li, Z. Zhang, {\it The homomorphisms and operations of rough groups}, The Scientific World Journal, 2014, Article ID:507972, 6 pages.

\bibitem{LXL2012} Z. Li, T. Xie, Q. Li, {\it Topological structure of generalized rough sets}, Comput. Math. with Appl., 63(2012), 1066--1071.

\bibitem{MH2005} D. Miao, S. Han, D. Li, L. Sun, {\it Rough group, rough subgroup and their properties}, D. \'{S}lkezak et al. (Eds.): RSFDGrC
2005, LNAI 3641, pp. 104每113, 2005.

\bibitem{N2014} C.A. Neelima, P. Isaac, {\it Rough anti-homomorphism on a rough group}, Global Math. Sci.: Theory and Practical, 6(2014), 79-每80.

\bibitem{P1982} Z. Pawlak, {\it Rough sets}, Int. J. Comput. Inf. Sci., 11(1982), 341--356.

\bibitem{PO1998} J. Pomykala, J.A. Pomykala, {\it The stone algebra of rough sets}, Bull. Polish Acad. Sci. Math., 36 (1998), 495--508.

\bibitem{RK2002} A.M. Radzikowska, E.E. Kerre, {\it A comparative study of fuzzy rough sets}, Fuzzy Sets and Systems, 126(2002), 137--155.

\bibitem{S2011} A.S. Salama, M.M.E. Abd El-Monsef, {\it New topological approach of rough set generalizations}, Int. J. Comput. Math.
8(7)(2011), 1347--1357.

\bibitem{WC2010} C.Z. Wang, D.G. Chen, {\it A short note on some properties of rough groups}, Comput. Math. Appl., 59(2010), 431--436.

\bibitem{WC2013} C.Z. Wang, D.G. Chen, {\it On rough approximations of groups}, Int. J. Mach. Learn. \& Cyber, 4(2013), 445--449.

\bibitem{WH2011} G.B. Wu, B. Huang, {\it Defect and revision of definition of rough group}, J. Nanjing Normal Univ. (Natural Sci. Edition), 34(2011), 39--42.

\bibitem{WM2019} W.Z. Wu, J.S. Mi, {\it The Mathematical Structure of Rough Set}, Science Press, 2019.

\bibitem{V20089} M. Vlach, Algebraic and Topological Aspects of Rough Set Theory, Fourth International Workshop on Computational Intelligence
Applications, IEEE SCM Hiroshima Chapter, Hiroshima University, Japan, December, 2008.

\bibitem{ZLL2019}  Y.L. Zhang, C.Q. Li, J.J. Li, {\it On characterizations of a pair of covering-based approximation operators}, Soft Computing, 23(2019), 3965--3972.

\bibitem{ZLL2016}  Y.L. Zhang, J.J. Li, C.Q. Li, {\it Topological structure of relation-based generalized rough sets}, Fund. Inf., 147(2016), 477--491.
\end{thebibliography}
\end{document}